\documentclass[11pt,bezier]{article}
\usepackage{amsmath,amssymb,amsfonts,euscript,graphicx}

\newtheorem{prethm}{{\bf Theorem}}

\newenvironment{thm}{\begin{prethm}{\hspace{-0.5
               em}{\bf.}}}{\end{prethm}}

\newtheorem{prepro}[prethm]{{\bf Theorem}}

\newtheorem{preprop}[prethm]{{\bf Proposition}}

\newtheorem{precor}[prethm]{{\bf Corollary}}

\newenvironment{cor}{\begin{precor}{\hspace{-0.5
               em}{\bf.}}}{\end{precor}}

\newtheorem{preconj}[prethm]{{\bf Conjecture}}

\newtheorem{preremark}[prethm]{{\bf Remark}}

\newenvironment{remark}{\begin{preremark}\rm{\hspace{-0.5
               em}{\bf.}}}{\end{preremark}}

\newtheorem{preexample}[prethm]{{\bf Example}}

\newenvironment{example}{\begin{preexample}\rm{\hspace{-0.5
               em}{\bf.}}}{\end{preexample}}

\newtheorem{preques}[prethm]{{\bf Question}}

\newtheorem{prelem}[prethm]{{\bf Lemma}}

\newenvironment{lem}{\begin{prelem}{\hspace{-0.5
               em}{\bf.}}}{\end{prelem}}

\newtheorem{prelam}{{\bf Lemma}}

\newtheorem{preproof}{{\bf Proof.}}

\newenvironment{proof}[1]{\begin{preproof}{\rm
               #1}\hfill{$\Box$}}{\end{preproof}}

\title{\bf \large  On the Cayley graph of a commutative ring \\with respect to its zero-divisors
\thanks {2010{ \it Mathematics Subject Classification}: 05C15, 05C25, 05C40, 05C69, 16N40.}}

\author{{\normalsize { G. Aalipour${}^{\mathsf{a,b}}$} and { S. Akbari${}^{\mathsf{a,b}}$}}\vspace{3mm}\\
{\footnotesize{${}^{\mathsf{a}}$ Department of Mathematical
 Sciences, Sharif University of Technology, Tehran, Iran}}\\
{\footnotesize{${}^{\mathsf{b}}$ School of Mathematics, Institute
for Research in Fundamental Sciences, (IPM),}}\\
{\footnotesize{}P.O. Box 19395-5746, Tehran, Iran}\\\\
{\footnotesize{$\mathsf{alipour\_ghodrat@mehr.sharif.ir}$\quad\quad
$\mathsf{s\_akbari@sharif.edu}$\quad\quad}}}
\date{}
\begin{document}
\maketitle
\begin{abstract}
{\small Let  $R$  be a commutative ring with unity and $R^{+}$ and $Z^*(R)$ be the additive group and the set of all non-zero zero-divisors of $R$, respectively. We denote by $\mathbb{CAY}(R)$ the Cayley graph $Cay(R^{+},Z^*(R))$. In this paper, we study $\mathbb{CAY}(R)$. Among other results, it is shown that for every zero-dimensional non-local ring $R$, $\mathbb{CAY}(R)$ is a connected graph of diameter $2$. Moreover, for a finite ring $R$, we obtain the vertex connectivity and the edge connectivity of $\mathbb{CAY}(R)$. We investigate rings $R$ with perfect $\mathbb{CAY}(R)$ as well. 
%As a consequence, we generalize a result on perfectness of unitary Cayley graph $Cay(R^{+},U(R))$. 
We also study $Reg(\mathbb{CAY}(R))$ the induced subgraph on the regular elements of $R$. This graph gives a family of  vertex transitive graphs. We show that if $R$ is a Noetherian ring and $Reg(\mathbb{CAY}(R))$ has no infinite clique, then $R$ is finite. Furthermore, for every finite ring $R$, the clique number and the chromatic number of $Reg(\mathbb{CAY}(R))$ are determined.}
\end{abstract}
%%%%%%%%%%%%%%%%%%%%%%%%%%%%%%%%%%%%%%%%%%%%%%%%%%%%%%%%%%%%%%%%%%%%%%%%%%%%%%%%%%%%%%%%%%%%%%%%%%%%%%%%%%%%%%%%%%%%%
\vspace{4mm} \noindent{\bf\large 1. Introduction}\vspace{4mm}\\
{Throughout this paper, all rings are assumed to be
commutative with unity. Let $R$ be a ring with the additive group $R^+$. We denote by $U(R)$, $Z(R)$, $Z^*(R)$, $Reg(R)$, ${\rm{Min}}(R)$, ${\rm{Spec}}(R)$ and ${\rm{Max}}(R)$,
the set of invertible elements, zero-divisors, non-zero zero-divisors, regular elements, minimal prime ideals,
prime ideals and maximal ideals of $R$,
respectively. The {\it{Jacobson radical}} and the {\it{nilradical}}
of $R$ are denoted by $J(R)$ and $Nil(R)$, respectively. The ring
$R$ is said to be \textit{reduced} if it has no non-zero nilpotent
element.  For a subset $X$ of $R$, by $(X)$, we mean \textit{the ideal generated by} $X$. The \textit{Krull dimension} of $R$ is denoted by $dim(R)$. %If $F$ is a field and $V$ is an $F$-vector space, then the dimension of $V$ is denoted by $v.dim_F(V)$. 
By $T(R)$, we mean the {\it{total ring}} of $R$
that is the ring of fractions of $R$ with respect to $R\setminus Z(R)$. A {\it{local ring}} is a ring with exactly one maximal ideal. A ring with finitely many maximal ideals is called a {\it{semi-local ring}}. A ring $R$ is said to
be a {\it{von Neumann regular ring}}, if for every $x\in R$, there exists $y\in R$ such that $x=x^2y$. The set of \textit{associated prime ideals of  $R$-module $R$} is denoted by ${\rm Ass}(R)=\{\mathfrak{p}\in {\rm{Spec}}(R)\, :\, \mathfrak{p}=Ann(x),\ \text{for some}\ x\in R\}$, where $Ann(x)=\{y\in R\,:\, xy=0\}$. For classical theorems and notations in commutative algebra, the interested reader is referred to \cite{ati} and \cite{sharp}.

Let $G$ be a graph with the vertex set $V(G)$. The \textit{complement} of $G$ is denoted by $\overline{G}$. 
For two vertices $x$ and $y$ in a graph $G$, a walk from $x$ to $y$ is a sequence $xe_1v_1\cdots v_{l-1}e_ly$, whose terms are alternately vertices and edges of $G$ (not necessarily distinct). We denote this walk by $x-\hspace{-.2cm}-v_1-\hspace{-.2cm}-\cdots -\hspace{-.2cm}-v_{l-1}-\hspace{-.2cm}-y$. The vertices $v_1,\ldots,v_{l-1}$ are called  \textit{internal vertices}. We say that two walks (paths) from $x$ to $y$ in a graph $G$ are \textit{vertex internally disjoint} if they share no common internal vertex. 
If $G$ is connected, then we mean by  $diam(G)$ and $d(x,y)$, the {\textit{diameter}} of $G$ and the \textit{distance} between two vertices $x$ and $y$. If $G$ is not connected, then $diam(G)$ is defined to be $\infty$. We denote by $K_n$ the \textit{complete graph} of
order $n$. The \textit{union} of two simple graphs $G$ and $H$ is the graph $G\cup H$ with the vertex set $V(G)\cup V (H)$ and the edge set $E(G)\cup E(H)$. If $G$ and $H$ are disjoint, we refer to their union as a \textit{disjoint union}, and generally denote it by $G+H$. The disjoint union of $n$ copies of $G$ is denoted by $nG$. We denote the {\textit {Cartesian product}} of two graphs $G$ and $H$ by  $G\Box H$. The \textit{direct product} (sometimes called \textit{Kronecker product} or \textit{tensor product}) of two graphs $G$ and $H$, denoted by $G\times H$, is a graph with the vertex set $V(G)\times V(H)$ and two distinct vertices $(x_1,y_1)$ and $(x_2,y_2)$ are adjacent if and only if $x_1$ and $x_2$ are adjacent in $G$ and $y_1$ and $y_2$ are adjacent in $H$.
A \textit{clique} in a graph $G$ is a subset of pairwise adjacent vertices. The supremum of the size of cliques in $G$, denoted by
$\omega(G)$, is called the \textit{clique number} of $G$. By
$\chi(G)$, we denote \textit{the chromatic number} of $G$ i.e. the
minimum number of colors which can be assigned to the vertices of
$G$ in such a way that every two adjacent vertices have different
colors. A coloring of the vertices such that any two adjacent vertices have different colors is called a \textit{proper vertex coloring}. A graph $G$ is called $k$-\textit{vertex colorable} if $G$ has a proper vertex coloring with $k$ colors. A graph $G$ is called \textit{perfect} if and only if for every finite induced subgraph $H$ of $G$, $\chi(H)=\omega(H)$. For $x\in V(G)$ we denote by $N(x)$ the set of all vertices of $G$ adjacent to $x$. Also, the size of $N(x)$ is denoted by $d(x)$ and is called the {\it{degree}} of $x$. The minimum degree of $G$ is denoted by $\delta(G)$. 
A graph is called \emph{$k$-regular}, if all its vertices have degree $k$. By $N[x]$, we mean $N(x)\cup \{x\}$. We denote by $\kappa(G)$ and $\kappa'(G)$, the \textit{vertex connectivity} and the \textit{edge connectivity} of $G$, respectively. For the definitions, see \cite{bondy}. 
%The \textit{energy} of a graph $G$ with $n$ vertices is the sum of the absolute values of all eigenvalues of its adjacency matrix and $G$ is called \textit{hyperenergetic} if its energy is at least $2n-1$. A graph $G$ is called \textit{integral}, if all eigenvalues of the adjacency matrix of $G$ are integer. 
A graph $G$ is called \textit{vertex transitive} (\textit{edge transitive}) if for every two vertices $x$ and $y$ (two edges $e$ and $e'$)  there exists $\rho\in Aut(G)$ such that $\rho(x)=y$ ($\rho(e)=e'$).

Let $G$ be a group with identity element $e$ and $\Omega$ be a non-empty subset of $G$ such that $e\notin \Omega$ and for every $g\in \Omega$, $g^{-1}\in \Omega$. The Cayley graph $Cay(G,\Omega)$ is a simple graph with the vertex set $G$ and two vertices $g$ and $h$ are adjacent if and only if $g^{-1}h\in \Omega$. A \textit{circulant graph} is a Cayley graph $Cay(\mathbb{Z}_n^+,\Omega)$, for some $\Omega\subseteq \mathbb{Z}_n\setminus\{0\}$ with property $\Omega=\{-x\,:\, x\in\Omega\}$.

Let $n$ be a positive integer, $\mathcal{D}=\{d\,:\, 1\leq d\leq n-1\, ,\, d\mid n\}$ and $T$ be a subset of $\mathcal{D}$. The \textit{gcd-graph} $X_n(T)$ has vertices $0,\ldots,n-1$ and two vertices $x$ and $y$ are adjacent if and only if $gcd(x-y,n)\in T$. The concept of gcd-graphs was first introduced by Klotz and Sander, see \cite{klotz-sander}. In \cite{so}, it is shown that integral circulant graphs are exactly the gcd-graphs. 
%Circulant graphs have various applications in the design of interconnection networks in parallel and distributed computing, while integral graphs play an important role in
%modeling quantum spin networks supporting the perfect state transfer. Integral circulant
%graphs also found applications in molecular graph energy. 
For more information on gcd-graphs, we refer the reader to \cite{Basic-Illic-clique}, \cite{Basic-Illic-chromatic}, \cite{klotz-sander} and \cite{so}.

The gcd-graph $X_n(\{1\})$ is called \textit{the unitary Cayley graph}, see \cite{fuchs} and \cite{klotz-sander} and references therein. In \cite{fuchs}, the unitary Cayely graph of a commutative ring $R$ is defined as $G_R=Cay(R^+,U(R))$. For more information on $G_R$, we refer the reader to \cite{akhtar}, \cite{kiani} and \cite{liu-zhou}. It is clear that $G_{\mathbb{Z}_n}\cong X_n(\{1\})$.

It is obvious that every gcd-graph $X_n(T)$ with the property $1\in T$, is of the form $X_n(\{1\})\cup X_n(T\setminus \{1\})$ and the gcd-graph $X_n(T)$ with $1\notin T$ is a subgraph of $X_n(\mathcal{D}\setminus \{1\})$, the complement of the unitary Cayley graph $X_n(\{1\})$. 
%Thus, we are motivated to study $X_n(\mathcal{D}\setminus \{1\})$. 
Aleksandar Ili$\text{\v{c}}$ in \cite{ilic}, determined the energy of $X_n(\mathcal{D}\setminus \{1\})$ and proved that this graph is hyperenergetic if and only if $n$ has at least two distinct prime factors and $n\neq2p$, where $p$ is a prime number. By generalizing the definition of $X_n(\mathcal{D}\setminus \{1\})\cong Cay(\mathbb{Z}_n^+, Z^*(\mathbb{Z}_n))$ to a commutative ring $R$, we study more properties of $X_n(\mathcal{D}\setminus \{1\})$. This generalization can be simply done by $\mathbb{CAY}(R)=Cay(R^+,Z^*(R))$, a graph whose vertices are elements of  $R$ and in which two distinct vertices  $x$  and  $y$  are joined by an edge if and only if $x - y\in Z(R)$. In Figure $1$, $\mathbb{CAY}(\mathbb{Z}_2\times\mathbb{Z}_2\times\mathbb{Z}_2)$ and $\mathbb{CAY}(\mathbb{Z}_6)$ are shown.
\\
\centerline{\includegraphics[scale=.3, angle=90, trim = 6cm 2cm 6cm 3cm, clip]{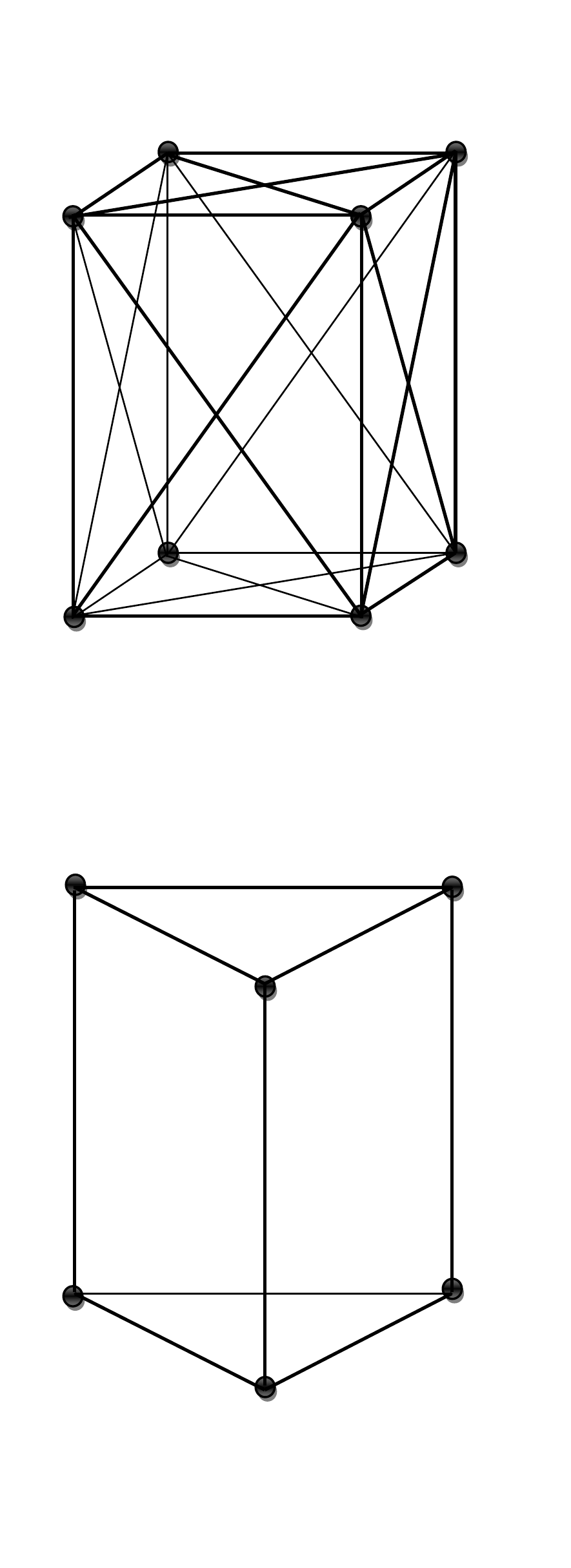}}\\
\centerline{$\text{\bf Figure 1. }$ $\mathbb{CAY}(\mathbb{Z}_2\times\mathbb{Z}_2\times\mathbb{Z}_2)$ and $\mathbb{CAY}(\mathbb{Z}_6)$}
\\

%In the recent years, assigning graphs to some algebraic structures has played an
%important role in both Algebra and Combinatorics. For
%example, the Cayley graphs of groups have very interesting combinatorial properties. Also, many papers were written in the recognition of the structure of a ring $R$ from its zero-divisor graph $\Gamma(R)$, for instance see \cite{anderson-livingston}. 
As a graph associated to a commutative ring $R$, the \textit{total graph} of $R$, denoted by $T(\Gamma(R))$, is a graph with the vertex set $R$ such that two distinct vertices $x$ and $y$ are adjacent if and only if $x+y\in Z(R)$. 
%This is called \textit{Cayley sum graph} (or sometimes \textit{addition Cayley graph}) $Cay^{+}(R^+,Z^*(R))$. 
The authors in \cite{totalgraph} have studied $T(\Gamma(R))$ and two of its subgraphs $Reg(\Gamma(R))$ and $Z(\Gamma(R))$, the induced subgraphs of $T(\Gamma(R))$ on $Reg(R)$ and $Z(R)$, respectively. However, $T(\Gamma(R))$ and $\mathbb{CAY}(R)$ are close in definition but sometimes they have different properties. For instance, the total graph of a finite ring $R$ can be biregular i.e. with exactly two distinct degrees but  we will see that $\mathbb{CAY}(R)$ is always regular. In \cite{shekarriz}, all finite rings $R$ such that $\mathbb{CAY}(R)\cong T(\Gamma(R))$ are characterized. 
%Thus, several  results of the current paper can be also applied to $T(\Gamma(R))$ for some finite rings $R$.

For a commutative Noetherian ring $R$, the chromatic number of $\mathbb{CAY}(R)$, as a simple graph associated with a commutative ring, was independently studied in \cite{japanese}. Some properties of $\mathbb{CAY}(R)$ such as clique number, independence number, domination number, girth, strongly regularity and edge transitivity have been  studied in \cite{Aali}. In  \cite{kiani}, Kiani et al. obtain eigenvalues of $\mathbb{CAY}(R)=\overline{G_R}$ as integers and compute the energy of $\mathbb{CAY}(R)$, for a finite ring $R$. In \cite{liu-zhou} among other results, the authors  characterize all finite rings $R$ such that $\mathbb{CAY}(R)$ is a Ramanujan graph.  In this paper, we continue studying $\mathbb{CAY}(R)$. We also introduce $Reg(\mathbb{CAY}(R))$ the induced subgraph of $\mathbb{CAY}(R)$ on $Reg(R)$. This graph gives a family of vertex transitive graphs. We first present some elementary results on $\mathbb{CAY}(R)$.
\begin{lem}\label{khavaseebtedaei}
Let $R$ be a ring. Then the following statements hold:

\noindent {\rm{(i)}} $\mathbb{CAY}(R)$ has no edge if and only if $R$ is an integral domain,\\
{\rm{(ii)}} If $(R,\mathfrak{m})$ is an Artinian local ring, then $\mathbb{CAY}(R)$ is a disjoint union of $|\frac{R}{\mathfrak{m}}|$
copies of the complete graph $K_{|\mathfrak{m}|}$,\\
{\rm{(iii)}} $\mathbb{CAY}(R)$ cannot be a complete graph,\\
{\rm{(iv)}} $\mathbb{CAY}(R)$ is vertex transitive,\\
{\rm{(v)}} $\mathbb{CAY}(R)$ is a regular graph of degree $|Z(R)|-1$ with isomorphic components.
\end{lem}
\begin{proof}
{Parts (i) and (iii) are obvious. Part (ii) follows from $Z(R)=\mathfrak{m}$. Part (iv) holds for every Cayley graph of a group. To prove the last part, note that under an automorphism of graph $G$, any component of $G$ is
isomorphically mapped to another component. Since $\mathbb{CAY}(R)$ is vertex-transitive, we conclude that the components of $\mathbb{CAY}(R)$ are isomorphic and so (v) is proved.}
\end{proof}
\begin{remark}
$\mathbb{CAY}(R)$ is vertex transitive but it is not necessarily edge transitive. To see this, consider $\mathbb{CAY}(\mathbb{Z}_6)\cong K_2\Box K_3$ which is not edge transitive (see Figure~$1$).
\end{remark}
%%%%%%%%%%%%%%%%%%%%%%%%%%%%%%%%%%%%%%%%%%%%%%%%%%%%%%%%%%%%%%%%%%%%%%%%%%%%%%%%%%%%%%%%%%%%%%%%%%%%%%%%%%%%%%%%%%%%%%%%%%%%%%%%%%%%
%%%%%%%%%%%%%%%%%%%%%%%%%%%%%%%%%%%%%%%%%%%%%%%%%%%%%%%%%%%%%%%%%%%%%%%%%%%%%%%%%%%%%%%%%%%%%%%%%%%%%%%%%%%%%%%%%%%%
%\section{The Connectivity of the Cayley Graph of a Ring}
\vspace{4mm} \noindent{\bf\large 2. The Connectivity of the Cayley Graph of a Ring}\vspace{4mm}\\
In this section, we study the connectivity of $\mathbb{CAY}(R)$. One of the main results of this section is: for every zero-dimensional non-local ring $R$, $\mathbb{CAY}(R)$ is a connected graph with diameter $2$. To prove this, we first need the following results.

The following lemma has a key role in our proofs. This lemma implies that every element of each minimal prime ideal of a ring is a zero-divisor, see \cite[Theorem 84]{kaplansky}.
\begin{lem}\label{minimalozerodivisors}
Let $R$ be a ring. Then the following statements hold:

\noindent {\rm{(i)}} If ${\mathfrak{p}}\in {\rm Min}(R)$ and $a\in \mathfrak{p}$, then there exists $b\in R\setminus \mathfrak{p}$ such that $ba^i=0$, for some $i\in \mathbb{N}$. In particular, $\bigcup_{{\mathfrak{p}}\in {\rm Min}(R)} {\mathfrak{p}} \subseteq Z(R)$.\\
{\rm{(ii)}} If $R$ is reduced, then $\bigcup_{{\mathfrak{p}}\in {\rm Min}(R)} {\mathfrak{p}}=Z(R)$.
\end{lem}
\begin{proof}
{Let $U=R\setminus \mathfrak{p}$, $V=\{1,a,a^2,\ldots \}$ and $S=UV$. Since $ \mathfrak{p}$ is a prime ideal, $U$ is a multiplicatively closed subset of $R$. So, $S$ is a multiplicatively closed subset of $R$. If $0 \notin S$, then by \cite[Theorem 3.44]{sharp}, there exists $\mathfrak{q}\in {\rm Spec}(R)$ such that $\mathfrak{q}\cap S=\varnothing$. This
yields that $\mathfrak{q}\subseteq R\setminus S\subseteq R\setminus U=\mathfrak{p}$. Thus, $\mathfrak{q}\subseteq\mathfrak{p}$. Since $a\in \mathfrak{p}\cap S$, we conclude that $\mathfrak{q}\subsetneq \mathfrak{p}$. This is a contradiction because ${\mathfrak{p}}\in {\rm Min}(R)$. So, we can assume
that $0 \in S$. Thus, there is a natural number $i$ such that $ba^i=0$,
for some $b\in R\setminus \mathfrak{p}$. This implies that $a\in Z(R)$ and the proof of Part (i) is complete. 
Now, we prove Part (ii). If $R$ is a reduced ring, then by \cite[Corollary 3.54]{sharp}, $\bigcap_{{\mathfrak{p}}\in {\rm Min}(R)} {\mathfrak{p}}=Nil(R)=\{0\}$. Therefore, for every $x\in Z(R)$, there exists ${\mathfrak{p}}\in {\rm Min}(R)$ such that $Ann(x)\nsubseteq {\mathfrak{p}}$. Thus, $x\in {\mathfrak{p}}$ and using Part (i), the proof is complete.}
\end{proof}

Using Part (i) of Lemma \ref{minimalozerodivisors}, we have the following corollary \cite[Theorem 91]{kaplansky}.
\begin{cor}\label{zerodivisorsofzero-dimensionals}
Let $R$ be a zero-dimensional ring. Then $Z(R)$ is the union of all maximal ideals of $R$. In particular, $Reg(R)=U(R)$ and $\mathbb{CAY}(R)=\overline{G_R}$.
\end{cor}

The following lemma is obvious and the proof is omitted.
\begin{lem}\label{zerodivisorsareideal}
Let $R$ be a ring such that $Z(R)$ is an ideal of $R$. Then $\mathbb{CAY}(R)$ is a disjoint union of $|\frac{R}{Z(R)}|$ complete graphs $K_{|Z(R)|}$ on the elements of a coset of $Z(R)$.
\end{lem}

So, to argue about the connectedness, we assume that $Z(R)$ is not an ideal.
\begin{thm}\label{diameter}
Let $R$ be a ring. Then $\mathbb{CAY}(R)$ is connected if and only if $R=(Z(R))$. Moreover, if $\mathbb{CAY}(R)$ is connected, then the following are equal:

\noindent {\rm{(i)}}  $d(0,1)$,\\
{\rm{(ii)}} $diam(\mathbb{CAY}(R))$,\\
{\rm{(iii)}} The least integer $n$ such that $1=z_1+\cdots+z_n$, where $z_i \in Z(R)$, for $1\leq i\leq n$.
\end{thm}
\begin{proof}
{The proof is very similar to the proof of \cite[Theorem 3.4]{totalgraph}. Suppose that $\mathbb{CAY}(R)$ is connected. Thus, there exists a path $1-\hspace{-.2cm}-z_1-\hspace{-.2cm}-\cdots-\hspace{-.2cm}-z_n-\hspace{-.2cm}-0$ from $1$ to $0$. Hence, $1=(1-z_1)+(z_1-z_2)+\cdots+z_n \in (Z(R))$.
Conversely, suppose that $1=a_1+\cdots+a_m$, where $a_i\in Z^*(R)$. Thus, $$\sum_{i=1}^m a_i-\hspace{-.2cm}-\sum_{i=1}^{m-1} a_i-\hspace{-.2cm}-\cdots-\hspace{-.2cm}-\sum_{i=1}^2a_i-\hspace{-.2cm}-a_1-\hspace{-.2cm}-0$$ is a walk of length $m$ from $1$ to $0$. Hence, for every vertex $x$ of $R$, 
$$x\sum_{i=1}^m a_i-\hspace{-.2cm}-x\sum_{i=1}^{m-1} a_i-\hspace{-.2cm}-\cdots-\hspace{-.2cm}-x\sum_{i=1}^2a_i-\hspace{-.2cm}-xa_1-\hspace{-.2cm}-0$$ is a walk of length $m$ from $x$ to $0$. So, there exists a path of length at most $m$ between $x$ and $0$. Thus, $\mathbb{CAY}(R)$ is connected. Moreover, $d(0,x)\leq d(0,1)$, for every vertex $x$. This argument also shows that $d(0,1)$ is equal to the least integer $n$ such that $1$ is a sum of $n$ elements of $Z(R)$. On the other hand, by the vertex transitivity of $\mathbb{CAY}(R)$, we conclude that
for every two distinct vertices $u$ and $v$, there exists $\varphi\in Aut(\mathbb{CAY}(R))$ such that $d(u,v)=d(0,\varphi(v))\leq d(0,1)$. Hence, $diam(\mathbb{CAY}(R))=d(0,1)$ and the proof is complete.}
\end{proof}
\begin{example}\label{equalitywithdiameteroftotalgraph}
By the above criterion and \cite[Theorem 3.4]{totalgraph}, we deduce that if $R=(Z(R))$, then $diam(\mathbb{CAY}(R))=diam\big(T(\Gamma(R))\big)$. For every integer $n\geq 2$, \cite[Example 3.8]{totalgraph}, provides a ring $R_n$ whose total graph has diameter $n$. Hence, for every integer $n\geq 2$, $diam(\mathbb{CAY}(R_n))=n$.
\end{example}

The next lemma has a key role in the proof of the main results of this section.
\begin{lem}\label{zerodivisorsofreducedsareshifted}
Let $R$ be a ring and $x\in R$.

\noindent {\rm (i)} If $x+Nil(R)\in Z(\frac{R}{Nil(R)})$, then $x\in Z(R)$. In particular, $diam(\mathbb{CAY}(R))\leq diam(\mathbb{CAY}(\frac{R}{Nil(R)}))$.\\
{\rm (ii)} Let $dim(R)=0$ or $R$ be a Noetherian ring with ${\rm Min}(R)={\rm Ass}(R)$. If $x\in Z(R)$, then $x+Nil(R)\in Z(\frac{R}{Nil(R)})$. In particular, $diam(\mathbb{CAY}(\frac{R}{Nil(R)}))= diam(\mathbb{CAY}(R))$.
\end{lem}
\begin{proof}
{First suppose that $x+Nil(R)\in Z(\frac{R}{Nil(R)})$. Thus, there exists $y\in R\setminus Nil(R)$ such that $xy\in Nil(R)$. So, there exists a positive integer $n$ such that $(xy)^n=0$. Since $y\in R\setminus Nil(R)$, there exists a non-negative integer $l$, $l<n$, such that
$x^{n-l}y^n=0$ and $x^{n-l-1}y^n\neq 0$. This implies that $x\in Z(R)$. \\To complete the proof of Part (i), note that if $$1+Nil(R)-\hspace{-.2cm}-z_1+Nil(R)-\hspace{-.2cm}-\cdots-\hspace{-.2cm}-z_n+Nil(R)-\hspace{-.2cm}-Nil(R),$$ is a path from $1+Nil(R)$ to $Nil(R)$ in
$\mathbb{CAY}(\frac{R}{Nil(R)})$, then $1-\hspace{-.2cm}-z_1-\hspace{-.2cm}-\cdots-\hspace{-.2cm}-z_n-\hspace{-.2cm}-0$ is a path from
 $1$ to $0$ in $\mathbb{CAY}(R)$. Now, by the argument in the proof of Theorem \ref{diameter}, $diam(\mathbb{CAY}(R))\leq diam(\mathbb{CAY}(\frac{R}{Nil(R)}))$.

We now prove the Part (ii). First suppose that $R$ is a Noetherian ring with ${\rm Min}(R)={\rm Ass}(R)=\{\mathfrak{p}_1,\ldots,\mathfrak{p}_n\}$.
By \cite[Corollary 9.36]{sharp}, $Z(R)=\cup_{i=1}^n\mathfrak{p}_i$.
Since $\frac{R}{Nil(R)}$ is a reduced ring, by Lemma \ref{minimalozerodivisors}, we conclude that $Z(\frac{R}{Nil(R)})= \cup_{i=1}^n\frac{\mathfrak{p}_i}{Nil(R)}$. Therefore, for every $x\in Z(R)$,  $x+Nil(R)\in Z(\frac{R}{Nil(R)})$.
Now, suppose that  $R$ is a ring with $dim(R)=0$. Let $x\in Z(R)$. Thus, by Corollary \ref{zerodivisorsofzero-dimensionals}, there exists $\mathfrak{p}\in {\rm Min}(R)$ such that $x\in \mathfrak{p}$. Hence, by Lemma \ref{minimalozerodivisors}, there exists $y\in R\setminus \mathfrak{p}$ such that $xy\in Nil(R)$. This implies that $x+Nil(R)\in Z(\frac{R}{Nil(R)})$. Now, similar to the last part of the proof of Part (i), the proof of Part (ii) is complete.}
\end{proof}
\begin{remark}\label{zerodivisorsofreductions}
Let $R$ be a Noetherian ring $R$ with ${\rm Min}(R) \neq {\rm Ass}(R)=\{\mathfrak{p}_1,\ldots,\mathfrak{p}_n\}$. Since ${\rm Min}(R) \subseteq {\rm Ass}(R)$, one may assume that ${\rm Min}(R)=\{\mathfrak{p}_1,\ldots,\mathfrak{p}_k\}$, for some $k$, $1\leq k< n$. By the Prime Avoidance Theorem (\cite[Theorem 3.61]{sharp}), let $x\in \mathfrak{p}_{k+1}\setminus \cup_{i=1}^k \mathfrak{p}_i$. We have the condition $x\in \cup_{i=1}^n \mathfrak{p}_i=Z(R)$ but $x+Nil(R)\notin \cup_{i=1}^k \frac{\mathfrak{p}_i}{Nil(R)}=Z(\frac{R}{Nil(R)})$. Hence, in  Part (ii) of the previous lemma, ${\rm Min}(R)={\rm Ass}(R)$ is not superfluous.
\end{remark}

Note that by Lemma \ref{zerodivisorsareideal} and Corollary \ref{zerodivisorsofzero-dimensionals}, if $R$ is a zero-dimensional local ring, then $\mathbb{CAY}(R)$ is a disjoint union of complete graphs. Now, we are in a position to prove the following result.
\begin{thm}\label{atleasttwominimaprimesareconnected}
Let $R$ be a zero-dimensional non-local ring. Then $\mathbb{CAY}(R)$ is connected and $diam(\mathbb{CAY}(R))=2$.
\end{thm}
\begin{proof}
{By Lemma \ref{zerodivisorsofreducedsareshifted}, one may assume that $R$ is a zero-dimensional reduced ring. Thus, by \cite[Exercise 22, p.64]{kaplansky}, $R$ is a von Neumann regular ring. Since $|{\rm{Min}}(R)|\geq 2$, $R$ has a non-zero zero-divisor, say $x$. Since $R$ is a von Neumann regular ring, there exists $y\in R$ such that $x=x^2y$ and so $e=xy$ is a non-zero idempotent of $R$. Note that $xe$ and $1-e$ are non-zero zero-divisors of $R$. Since $u=xe+1-e$ is a unit of $R$ with inverse $ye+1-e$, we conclude that $1$ is a sum of two zero-divisors of $R$.  Thus, by Theorem \ref{diameter}, $diam(\mathbb{CAY}(R))\le2$. As $\mathbb{CAY}(R)$ is not a complete graph, we have $diam(\mathbb{CAY}(R))=2$. The assertion is proved.}
\end{proof}

The following result is a generalization of \cite[Theorem 3.4]{totalgraph}. The proof simply follows from Theorem \ref{diameter}, Example \ref{equalitywithdiameteroftotalgraph} and the previous result.
\begin{cor}\label{connectedness-of-total-graph}
Let $R$ be a zero-dimensional non-local ring. Then $T(\Gamma(R))$ is connected and $diam(T(\Gamma(R)))=2$.
\end{cor}

Till now, we have studied the connectedness of $\mathbb{CAY}(R)$ for every zero-dimensional ring $R$. The following example shows that Theorem \ref{atleasttwominimaprimesareconnected} does not hold for  Noetherian rings with at least two minimal prime ideals. Moreover, it shows that the assumptions in the previous theorem are necessary.

\begin{example}
Let $K$ be a field and $R=\frac{K[x,y]}{(xy)}$. Clearly, $R$ is a Noetherian
reduced ring with $dim(R)=1$, ${\rm Min}(R)=\{(x+(xy)),(y+(xy))\}$ and by Lemma \ref{minimalozerodivisors}, $Z(R)=(x+(xy))\cup(y+(xy))$. Thus, $\mathbb{CAY}(R)$ is a disjoint union of $|K|$ connected graphs of diameter $2$ with the vertex sets $\Gamma_{\alpha}=\{\alpha+f(x)+g(y)+(xy)\,:\, f\in K[x], g\in K[y]\ {\rm{and}}\  f(0)=g(0)=0\}$, where $\alpha\in K$.
%On the other hand, for every two rings $R_1$ and $R_2$, $\mathbb{CAY}(R_1\times R_2)$ is a connected graph with diameter $2$. Thus, for $R=\mathbb{Z}_4\times K[x_1,\ldots,x_{n}]$, which is a Noetherian non-reduced ring, we have $dim(R)=n$, $|{\rm Min}(R)|=2$ and $diam(\mathbb{CAY}(R))=2$.
Note that for every $\alpha\in K$, the ideals $\frac{(x-\alpha,y)}{(xy)}$ and $\frac{(x,y-\alpha)}{(xy)}$ are distinct maximal ideals of $R$. So, $R$ is a non-local ring. 
\end{example}

In the sequel, for a finite  ring $R$, we obtain the vertex connectivity and edge connectivity of $\mathbb{CAY}(R)$. We first need the following lemmas.
\begin{lem}\label{shiftingbynilpotents}
Let $R$ be a ring, $x\in R$ and $a\in Nil(R)$. Then $x+a\in Z(R)$ if and only if $x\in Z(R)$.
\end{lem}
\begin{proof}
{If $a=0$, then the assertion is obvious. Thus, assume that $a\neq 0$. First suppose that there exists $0\neq y\in R$ such that $xy=0$. Since $a$ is nilpotent, there exists a positive integer $n$ such that $a^n=0$ and $a^{n-1}\neq 0$. Let $k$, $k\leq n$, be the biggest positive integer such that $ya^{n-k}\neq 0$ and $ya^{n-k+1}=0$. Clearly, $ya^{n-k}(x+a)=0$ and so $x+a\in Z(R)$. Conversely, suppose that $x+a\in Z(R)$. Thus, $x=x+a+(-a)\in Z(R)$, as desired.}
\end{proof}

Let $G$ be a connected graph. A non-empty subset $S$ of vertices of $G$ is called a \textit{vertex cut} of $G$ if $G-S$ (the removal of vertices of $S$ from $G$) is not connected or has exactly one vertex. Note that by Menger's Theorem, for a finite connected graph $G$, $\kappa(G)$ is equal to the minimum size of vertex cuts of $G$ (see \cite[Theorem 4.2.21]{west}).
\begin{lem}\label{connectivityofproductoffields}
Let $n\geq 2$ be a positive integer and $F_1,\ldots,F_n$ be finite fields. Then $$\kappa(\mathbb{CAY}(F_1\times \cdots\times F_n))\geq |Z(F_1\times \cdots\times F_n)|-1.$$
\end{lem}
\begin{proof}
{We show that for every two distinct vertices $X$ and $Y$ of $\mathbb{CAY}(F_1\times \cdots\times F_n)$, there are at least $|Z(F_1\times \cdots\times F_n)|-1$ pairwise vertex internally disjoint paths (or simply pairwise internally disjoint paths) from $X$ to $Y$. We prove this by induction on $n$. By \cite{spacapan}, $\kappa(K_{|F_1|}\Box K_{|F_2|})=|F_1|+|F_2|-2$. Note that $|Z(F_1\times F_2)|-1=|(F_1\times F_2)\setminus (F_1^*\times F_2^*)|-1=|F_1|+|F_2|-2$, where $F_i^*=F_i\setminus\{0\}$. Since  $\mathbb{CAY}(F_1\times F_2)\cong K_{|F_1|}\Box K_{|F_2|}$, for $n=2$ the assertion holds. Now, let the assertion hold for $n$ and let $X=(x_1,\ldots,x_{n+1})$, $Y=(y_1,\ldots,y_{n+1})\in F_1\times \cdots\times F_{n+1}$. First we need some notations. We recursively express $X$ and $Y$ by $X=(x_1,\widehat{X})$ and $Y=(y_1,\widehat{Y})$, where $\widehat{X}=(x_2,\ldots,x_{n+1})\in F_2\times \cdots\times F_{n+1}$ and $\widehat{Y}=(y_2,\ldots,y_{n+1})\in F_2\times \cdots\times F_{n+1}$, respectively. 

If $\widehat{X}\neq \widehat{Y}$, by induction hypothesis, there exist $|Z(F_2\times \cdots\times F_{n+1})|-1$ pairwise internally disjoint paths from $\widehat{X}$ to $\widehat{Y}$ in $\mathbb{CAY}(F_2\times \cdots\times F_{n+1})$. It means that for every vertex $U$ adjacent to $\widehat{X}$ in $\mathbb{CAY}(F_2\times \cdots\times F_{n+1})$, there exists a unique path $P(U)$ (among those $|Z(F_2\times \cdots\times F_{n+1})|-1$ paths) which contains $U$. Let $P_t(U)$ be that terminal vertex of $P(U)$ which is adjacent to $\widehat{Y}$. So, $P(U)$ has the following form. 
$$P(U): \widehat{X}-\hspace{-.2cm}-U-\hspace{-.2cm}-\cdots-\hspace{-.2cm}-P_t(U)-\hspace{-.2cm}-\widehat{Y}.$$
\noindent Note that $U\neq \widehat{X}$ and $P_t(U)\neq \widehat{Y}$. If $U=P_t(U)$, we simply mean $\widehat{X}-\hspace{-.2cm}-U-\hspace{-.2cm}-\widehat{Y}$.

Now, we have the required notations to show that there exist at least $|Z(F_1\times \cdots\times F_{n+1})|-1$ pairwise  internally disjoint paths from $X$ to $Y$. We need to consider the two following cases.

\noindent
{\rm \bf {Case 1. }} $X$ and $Y$ have at least an equal component, say the first component.

Let $X=(a,\widehat{X})$, $Y=(a,\widehat{Y})$. Note that in this case $\widehat{X}\neq \widehat{Y}$. Thus,  $X-\hspace{-.2cm}-(a,A)-\hspace{-.2cm}-Y$ is a path from $X$ to $Y$, for every $A\in F_2\times \cdots\times F_{n+1}\setminus \{\widehat{X}, \widehat{Y}\}$. Every two paths of this form are internally disjoint. Since $X$ and $Y$ are adjacent, we find $|F_2\times \cdots\times F_{n+1}|-1$ pairwise  internally disjoint paths from $X$ to $Y$ of the following types. 

{\rm \bf {Type 1.1. }} The single path $X-\hspace{-.2cm}-Y$ of length $1$.

{\rm \bf {Type 1.2. }} The paths of the form $X-\hspace{-.2cm}-(a,A)-\hspace{-.2cm}-Y$ of length $2$, where $A\in F_2\times \cdots\times F_{n+1}\setminus \{\widehat{X}, \widehat{Y}\}$.

Since $\widehat{X}\neq \widehat{Y}$, we are allowed to use notation $P_t(U)$ for every vertex $U$ adjacent to $\widehat{X}$. Now, let $b\in F_1\setminus \{a\}$. We consider the following type of  paths from $X$ to $Y$:

{\rm \bf {Type 1.3. }} The paths of the form $X-\hspace{-.2cm}-(b,U)-\hspace{-.2cm}-(b,P_t(U))-\hspace{-.2cm}-Y$, where $b\in F_1\setminus \{a\}$ and  $U$ is a vertex adjacent to $\widehat{X}$ in  $\mathbb{CAY}(F_2\times \cdots\times F_{n+1})$. 

If $U=P_t(U)$, we simply mean $X-\hspace{-.2cm}-(b,U)-\hspace{-.2cm}-Y$. Since for every two distinct vertices $U$ and $V$ which both are  adjacent to $\widehat{X}$, we have $\{U,P_t(U)\}\cap\{V,P_t(V)\}=\varnothing$, we conclude that all the paths of  this type are pairwise  internally disjoint. The number of paths of type $3$ is $(|F_1|-1)(|Z(F_2\times \cdots\times F_{n+1})|-1)$. 

Finally, we consider the next type of the paths. 

{\rm \bf {Type 1.4. }} The paths of the form $X-\hspace{-.2cm}-(b,\widehat{X})-\hspace{-.2cm}-(b,\widehat{Y})-\hspace{-.2cm}-Y$ of length $4$, where $b\in F_1\setminus \{a\}$. 
 
 It is clear that the paths of this type form $|F_1|-1$  pairwise  internally disjoint paths from $X$ to $Y$.  
 
Note that by the construction we give, the paths of the same type are pairwise  internally disjoint. Moreover, two paths of different types are internally disjoint. Thus, from these $4$ types of paths we obtain 
$$(|F_1|-1)|Z(F_2\times \cdots\times F_{n+1})|+|F_2\times \cdots\times F_{n+1}|-1$$
pairwise  internally disjoint paths from $X$ to $Y$. Since 
$$Z(F_1\times \cdots\times F_{n+1})=\Big(\hspace{-3mm}\bigcup_{\hspace{3mm}x\in F_1\setminus\{0\}}\hspace{-4mm}\{x\}\times Z(F_2\times \cdots\times F_{n+1})\Big)\bigcup\, \{0\}\times F_2\times \cdots\times F_{n+1},$$
we deduce that $$(|F_1|-1)|Z(F_2\times \cdots\times F_{n+1})|+|F_2\times \cdots\times F_{n+1}|-1=|Z(F_1\times \cdots\times F_{n+1})|-1.$$
Hence, we have $|Z(F_1\times \cdots\times F_{n+1})|-1$ pairwise internally disjoint paths from $X$ to $Y$ in this case. 

\noindent
{\rm \bf {Case 2. }} For $i=1,\ldots, n+1$, $x_i\neq y_i$.

We show that we can assume that every $F_i\cong\mathbb{Z}_2$. Suppose that one of the $F_i$, say $F_1$, has at least $3$ elements. Similar to the previous case, we consider the following types of paths from $X$ to $Y$.

{\rm \bf {Type 2.1. }} The paths of the form $X-\hspace{-.2cm}-(b,U)-\hspace{-.2cm}-(b,P_t(U))-\hspace{-.2cm}-Y$, where $b\in F_1\setminus \{x_1,y_1\}$ and  $U$ is a vertex adjacent to $\widehat{X}$ in  $\mathbb{CAY}(F_2\times \cdots\times F_{n+1})$. 

{\rm \bf {Type 2.2. }} The paths of the form $X-\hspace{-.2cm}-(b,\widehat{X})-\hspace{-.2cm}-(b,\widehat{Y})-\hspace{-.2cm}-Y$ of length $4$, where $b\in F_1\setminus \{x_1,y_1\}$. 

These two types of paths give 
$$(|F_1|-2)\big(|Z(F_2\times \cdots\times F_{n+1})|-1\big)+|F_1|-2=(|F_1|-2)|Z(F_2\times \cdots\times F_{n+1})|$$ 
pairwise  internally disjoint paths from $X$ to $Y$ with the internal vertices in $(F_1\setminus\{x_1,y_1\})\times F_2\times \cdots\times F_{n+1}$. 

Hence, every possible path from $X$ to $Y$ with vertices in $\{x_1,y_1\}\times F_2\times \cdots\times F_{n+1}$ is internally disjoint from the paths in Types 2.1 and 2.2. Thus, if we find $|Z(F_1\times \cdots\times F_{n+1})|-1-(|F_1|-2)|Z(F_2\times \cdots\times F_{n+1})|$ pairwise internally disjoint paths from $X$ to $Y$ with vertices in $\{x_1,y_1\}\times F_2\times \cdots\times F_{n+1}$, then we are done. To see this, we first consider the following claim. 

{\rm \bf {Claim. }} $|Z(F_1\times \cdots\times F_{n+1})|-1-(|F_1|-2)|Z(F_2\times \cdots\times F_{n+1})|$ is the number of neighbors of $X$ in $\{x_1,y_1\}\times F_2\times \cdots\times F_{n+1}$. 

To prove the claim, let $N_1$ be the set of neighbors of $X$ in $(F_1\setminus\{x_1,y_1\})\times F_2\times \cdots\times F_{n+1}$ and $N_2$ be the set of neighbors of $X$ in $\{x_1,y_1\}\times F_2\times \cdots\times F_{n+1}$. Clearly $|Z(F_1\times \cdots\times F_{n+1})|-1=|N_1|+|N_2|$ and $|N_1|=(|F_1|-2)|Z(F_2\times \cdots\times F_{n+1})|$. This proves the claim.

Hence, to complete the proof, we should find $|N_2|$ pairwise internally disjoint paths from $X$ to $Y$ whose vertices are in $\{x_1,y_1\}\times F_2\times \cdots\times F_{n+1}$.  It is not hard to check that $|N_2|$ is degree of $(1,\widehat{X})$ in $\mathbb{CAY}(\mathbb{Z}_2\times F_2\times \cdots\times F_{n+1})$. Hence, it suffices to find at least $|Z(\mathbb{Z}_2\times F_2\times \cdots\times F_{n+1})|-1$ pairwise  internally disjoint paths from $(1,\widehat{X})$ to $(0,\widehat{Y})$ in $\mathbb{CAY}(\mathbb{Z}_2\times F_2\times \cdots\times F_{n+1})$. Thus, we may assume that $|F_1|=2$ and by continuing this procedure we can suppose that $F_i\cong \mathbb{Z}_2$, for  $i=1,\ldots,n+1$ and $X=(1,\ldots,1)$, $Y=(0,\ldots,0)$. Since for every $z\in Z^*(\mathbb{Z}_2\times\cdots\times \mathbb{Z}_2)$, $(1,\ldots,1)-\hspace{-.2cm}-z-\hspace{-.2cm}-(0,\ldots,0)$ is a path from $(1,\ldots,1)$ to $(0,\ldots,0)$, we obtain $|Z(\mathbb{Z}_2\times\cdots\times \mathbb{Z}_2)|-1$ pairwise  internally disjoint paths from $(1,\ldots,1)$ to $(0,\ldots,0)$. Hence, by the above argument we can find $|Z(F_1\times \cdots\times F_{n+1})|-1$ pairwise internally disjoint paths from $X$ to $Y$. 

This implies that the assertion holds for $n+1$ in both Cases 1 and 2. So by induction the proof is complete.}
\end{proof}
\begin{thm}{\rm {(\cite[Theorem 9.14]{bondy})}}\label{edge-connectivity-vs-vertex-transitivity}
Let $G$ be a simple connected vertex transitive graph of positive degree $d$. Then $\kappa'(G)=d$.
\end{thm}
Now, we are in a position to obtain the vertex connectivity and the edge connectivity of $\mathbb{CAY}(R)$. A well-known theorem duo to Watkins (see \cite[Corollary 1.A]{watkins}) states that the vertex connectivity of every connected edge transitive graph $G$ equals to $\delta(G)$. In \cite{akhtar}, the authors apply this theorem to obtain the vertex connectivity of the unitary Cayley graph $G_R$. Unfortunately,  $\mathbb{CAY}(R)$ is not necessarily edge transitive. In \cite{Aali}, all finite rings $R$ whose  $\mathbb{CAY}(R)$ is edge transitive are characterized. The following result states that $\mathbb{CAY}(R)$ is a reliable network i.e. the vertex connectivity of $\mathbb{CAY}(R)$ equals to degree of regularity,  for every finite reduced ring $R$. Since for every finite graph $G$, we have $\kappa(G)\leq\kappa'(G)\leq \delta(G)$ (see \cite[Theorem 4.1.9]{west} or \cite[Exercise 9.3.2]{bondy}), we deduce that $\mathbb{CAY}(R)$ gives a class of vertex transitive graphs with optimal connectivity. 
%In \cite{watkins}, vertex transitive graphs whose vertex connectivity is less than degree of regularity are studied. 

\begin{thm}\label{vertex-edge-connect-no}
Let $R$ be a finite non-local ring. Then
$\kappa(\mathbb{CAY}(R))=|Z(R)|-|Nil(R)|$ and $\kappa'(\mathbb{CAY}(R))=|Z(R)|-1.$
\end{thm}
\begin{proof}
{Since $R$ is non-local, by Theorem \ref{atleasttwominimaprimesareconnected}, $\mathbb{CAY}(R)$ is connected. Thus, by Theorem \ref{edge-connectivity-vs-vertex-transitivity} we have  $\kappa'(\mathbb{CAY}(R))=|Z(R)|-1$. Now, we prove that $\kappa(\mathbb{CAY}(R))=|Z(R)|-|Nil(R)|$. 

First, we show that $\kappa(\mathbb{CAY}(R))\geq |Z(R)|-|Nil(R)|$. To see this, we prove that for every two distinct vertices $x$ and $y$ of $\mathbb{CAY}(R)$, there are at least $|Z(R)|-|Nil(R)|$ pairwise (vertex) internally disjoint paths from $x$ to $y$. First suppose that $x-y\in Nil(R)$. Thus, by Lemma \ref{shiftingbynilpotents}, for every $z\in Z^*(R)\setminus \{x-y\}$, $x-\hspace{-.2cm}-z+y-\hspace{-.2cm}-y$ is a path of length $2$ from $x$ to $y$. Since $x$ and $y$ are adjacent, we deduce that there exist at least $|Z(R)|-1\geq |Z(R)|-|Nil(R)|$ pairwise  internally disjoint paths from $x$ to $y$. Now, assume that  $x-y\notin Nil(R)$. Clearly, $\frac{R}{Nil(R)}$ is non-local and by \cite[Theorem 8.7]{ati} (or Chinese Remainder Theorem \cite[Proposition 1.10]{ati}), $\frac{R}{Nil(R)}$ is a finite product of finite fields. Hence, by Lemma \ref{connectivityofproductoffields}, there exist $|Z(\frac{R}{Nil(R)})|-1$ pairwise  internally disjoint paths from $x+Nil(R)$ to $y+Nil(R)$ in $\mathbb{CAY}(\frac{R}{Nil(R)})$. This number is exactly the number of neighbors of $x+Nil(R)$ in $\mathbb{CAY}(\frac{R}{Nil(R)})$. For every neighbor of $x+Nil(R)$ say $x+z+Nil(R)$ let $P(z)$ be the unique  path which contains $x+z+Nil(R)$, where $z+Nil(R)\in Z^*(\frac{R}{Nil(R)})$. Assume that   $P(z)$ has the following vertices.
 $$x+Nil(R)-\hspace{-.2cm}-x_1^{(z)}+Nil(R)-\hspace{-.2cm}-\cdots-\hspace{-.2cm}-x_{k(z)}^{(z)}+Nil(R)-\hspace{-.2cm}-y+Nil(R),$$ where $x_1^{(z)}+Nil(R)=x+z+Nil(R)$. Then by Lemma \ref{zerodivisorsofreducedsareshifted}, for every $m\in Nil(R)$, $$x-\hspace{-.2cm}-x_1^{(z)}+m-\hspace{-.2cm}-\cdots-\hspace{-.2cm}-x_{k(z)}^{(z)}+m-\hspace{-.2cm}-y$$ 
forms a path from $x$ to $y$ in $\mathbb{CAY}(R)$. We denote this path by $P(z)+m$. As $z$ ranges over $Z^*(\frac{R}{Nil(R)})$ and $m$ ranges over $Nil(R)$, the paths $P(z)+m$ generate pairwise  internally disjoint paths from $x$ to $y$ in $\mathbb{CAY}(R)$. Thus, we obtain at least $|Nil(R)|(|Z(\frac{R}{Nil(R)})|-1)$ pairwise  internally disjoint paths from $x$ to $y$ in $\mathbb{CAY}(R)$. By Lemma \ref{zerodivisorsofreducedsareshifted}, we have $|Nil(R)|(|Z(\frac{R}{Nil(R)})|-1)=|Z(R)|-|Nil(R)|$. 
%Since for every two distinct $z_1+Nil(R)$ and $z_2+Nil(R)$ in $Z^*(\frac{R}{Nil(R)})$, the corresponding paths $P(z_1)$ and $P(z_2)$ are  internally disjoint in $\mathbb{CAY}(\frac{R}{Nil(R)})$, we conclude that for every two distinct elements $m_1\in Nil(R)$ and $m_2\in Nil(R)$, the new paths $P(z_1)+m_1$ and $P(z_2)+m_2$ are  internally disjoint paths in $\mathbb{CAY}(R)$. 
Hence, we obtain at least $|Z(R)|-|Nil(R)|$ pairwise  internally disjoint paths from $x$ to $y$ in $\mathbb{CAY}(R)$. Therefore, $\kappa(\mathbb{CAY}(R))\geq |Z(R)|-|Nil(R)|$. 

Now, we show that $\kappa(\mathbb{CAY}(R))\leq |Z(R)|-|Nil(R)|$. To see this, it suffices to prove that $Z(R)\setminus Nil(R)$ is a vertex cut of $\mathbb{CAY}(R)$. Let $H$ be the graph obtained from $\mathbb{CAY}(R)$ by removing the vertices in $Z(R)\setminus Nil(R)$. Note that the vertices of $H$ consist of $Reg(R)\cup Nil(R)$. Now, Lemma \ref{shiftingbynilpotents} implies that no vertex of $Reg(R)$ has a neighbor in $Nil(R)$ in the graph $H$. Thus, $Nil(R)$ forms a non-empty (complete) connected  component of graph $H$. Since $\mathbb{CAY}(R)$ is connected, we deduce that $Z(R)\setminus Nil(R)$ is a vertex cut. Thus, by Menger's Theorem $\kappa(\mathbb{CAY}(R))\leq |Z(R)|-|Nil(R)|$. The assertion is proved.}
\end{proof}
\begin{cor}
Let $R$ be a finite non-local ring. If $R$ has a residue field isomorphic to $\mathbb{Z}_2$, then $\kappa'\big(T(\Gamma(R))\big)=|Z(R)|-1$ and $\kappa\big(T(\Gamma(R))\big)=|Z(R)|-|Nil(R)|$.
\end{cor}
\begin{proof}
{Since every finite ring decomposes into a  product of finite local rings (see \cite[Theorem 8.7]{ati}), by Part (b) of \cite[Theorem 5.2]{shekarriz}, we obtain that  $T(\Gamma(R))\cong\mathbb{CAY}(R)$. Now, the assertion follows from Theorem \ref{vertex-edge-connect-no}. }
\end{proof}
\begin{thm}
Let $R$ be a finite non-local ring. Then $\mathbb{CAY}(R)$ is Hamiltonian.
\end{thm}
\begin{proof}
{Since $R$ is non-local, $|R|\geq 3$. Moreover, by Theorem \ref{atleasttwominimaprimesareconnected}, $\mathbb{CAY}(R)$ is connected. Thus, by \cite[Corollary 3.2]{marusic}, $\mathbb{CAY}(R)$ is Hamiltonian.}
\end{proof}
%%%%%%%%%%%%%%%%%%%%%%%%%%%%%%%%%%%%%%%%%%%%%%%%%%%%%%%%%%%%%%%%%%%%%%%%%%%%%%%%%%%%%%%%%%%%%%%%%%%%%%%%%%%%%%%%%%%%%%%%%%%%%%%%%%%%%%
%%%%%%%%%%%%%%%%%%%%%%%%%%%%%%%%%%%%%%%%%%%%%%%%%%%%%%%%%%%%%%%%%%%%%%%%%%%%%%%%%%%%%%%%%%%%%%%%%%%%%%%%%%%%%%%%%%%%%%%%%%%%%%%%%%%%%%%%%%%%%%%%%%%%%%
\vspace{4mm} \noindent{\bf\large 3. The Quotient and the Perfectness of the Cayley Graph of a Ring}\vspace{4mm}\\
In this section, we study the quotient graph of $\mathbb{CAY}(R)$ and  its relation with the quotient ring $\frac{R}{Nil(R)}$. In the end, we characterize zero-dimensional semi-local rings whose Cayley graph is perfect. First, for a graph $G$, we provide some background on the \textit{quotient graph} $G/S$, whose properties are tightly close to those of $G$. Define the relation $\thicksim$ on the vertices of $G$ as follows: $x\thicksim y$ if and only if $N[x]=N[y]$. It is an equivalence relation on the vertices of $G$. Denote the equivalence class of $x$ by $[x]$ and define a simple graph $G/S$ with the vertex set $\{[x]\,:\, x\in G\}$ and two distinct vertices $[x]$ and $[y]$ are adjacent if and only if $x$ and $y$ are adjacent in $G$. This graph is independent of chosen representatives and it is well-defined. For more information we refer the interested reader to \cite{productgraphs}. In the sequel, we would like to determine the equivalence class of each vertex of $\mathbb{CAY}(R)$.
\begin{thm}\label{closedneighborhood}
Let $R$ be a ring and $Z(R)$ be a union of finitely many minimal prime ideals of $R$. Then for every two elements $x$ and $y$ of $R$, the following statements are equivalent:

\noindent {\rm (i)}  $x\thicksim y$,\\
{\rm (ii)} $N[x]=N[y]$,\\
{\rm (iii)} $x-y\in Nil(R)$.
\end{thm}
\begin{proof}
{The equivalence of (i) and (ii) is clear from the definition. Now, we show that (ii) and (iii) are equivalent. By Lemma \ref{shiftingbynilpotents}, (iii) $\Longrightarrow$ (ii). Now, suppose that $N[x]=N[y]$ and $Z(R)=\cup_{i=1}^n{\mathfrak{p}}_i$, where ${\mathfrak{p}}_1,\ldots,{\mathfrak{p}}_n\in {\rm Min}(R)$. We claim that ${\rm Min}(R)=\{{\mathfrak{p}}_1,\ldots,{\mathfrak{p}}_n\}$. Let $\mathfrak{p}\in {\rm Min}(R)$. By Lemma \ref{minimalozerodivisors}, we have $\mathfrak{p}\subseteq Z(R)=\cup_{i=1}^n{\mathfrak{p}}_i$. Therefore, by  Prime Avoidance Theorem (\cite[Theorem 3.61]{sharp}), there exists $1\leq j\leq n$ such that $\mathfrak{p}\subseteq  {\mathfrak{p}}_j$. Since both $\mathfrak{p}$ and $\mathfrak{p}_j$ are minimal prime ideals, we deduce that $\mathfrak{p}={\mathfrak{p}}_j$. Thus, ${\rm Min}(R)\subseteq \{{\mathfrak{p}}_1,\ldots,{\mathfrak{p}}_n\}$  and the claim is proved. 
%So, we conclude  that $Z(R)=\cup_{{\mathfrak{p}}\in {\rm Min}(R)} {\mathfrak{p}}$. 
Let $A=\{i\,:\, x-y\in {\mathfrak{p}}_i\}$. Since $N[x]=N[y]$, $A\neq\emptyset$. If $A^c=\{1,\ldots,n\}\setminus A\neq\emptyset$, by Prime Avoidance Theorem,  there exists $z\in \cap_{i\in A^c}{\mathfrak{p}_i}$ such that $z\notin \cup_{i\in A}{\mathfrak{p}_i}$. Now, we show that $z-x+y$ is a regular element of $R$. If $z-x+y\in Z(R)=\cup_{i=1}^n{\mathfrak{p}}_i$, then for some $1\leq k\leq n$, $z-x+y\in {\mathfrak{p}}_k$. If $k\in A$, then $z\in {\mathfrak{p}}_k$, a contradiction. So, $k\in A^c$. This yields that $x-y\in {\mathfrak{p}}_k$, a contradiction again. Hence, $z-x+y$ is a regular element of $R$. Thus, $z+y$ is adjacent to $x$ in $\overline{\mathbb{CAY}(R)}$. Now, $N[x]=N[y]$ implies $z+y$ is also adjacent to $y$ in $\overline{\mathbb{CAY}(R)}$. Hence, $z$ is a regular element of $R$, a contradiction. Therefore, $A^c=\emptyset$ and so $x-y\in \cap_{i=1}^{n}{\mathfrak{p}}_i$. Since ${\rm Min}(R)=\{{\mathfrak{p}}_1,\ldots,{\mathfrak{p}}_n\}$, we deduce that $x-y\in\cap_{i=1}^{n}{\mathfrak{p}}_i=Nil(R)$. The proof is complete.}
\end{proof}
\begin{remark}\label{zerodivisorsaunioofminimalprimes}
Zero-dimensional semi-local rings, reduced rings with finitely many minimal prime ideals and Noetherian rings such that every associated prime ideal is a minimal prime ideal are examples of rings satisfying the assumptions of Theorem~\ref{closedneighborhood}.
\end{remark}
\begin{cor}\label{modulojacobson}
Let $R$ be a zero-dimensional semi-local ring. Then $\mathbb{CAY}(R)/S\cong\mathbb{CAY}(\frac{R}{J(R)})$.
\end{cor}
\begin{proof}
{By Theorem \ref{closedneighborhood}, for every $x\in R$, $[x]=x+Nil(R)=x+J(R)$. Thus, the vertex set of $\mathbb{CAY}(R)/S$ is $\frac{R}{J(R)}$. Also, note that for every $z\in R$, by Lemma \ref{zerodivisorsofreducedsareshifted}, $z\in Z(R)$ if and only if $z+J(R)\in Z(\frac{R}{J(R)})$ and so two distinct vertices of $\mathbb{CAY}(R)/S$ are adjacent in $\mathbb{CAY}(R)/S$ if and only if they are adjacent in $\mathbb{CAY}(\frac{R}{J(R)})$. The proof is complete.}
\end{proof}

The \textit{Strong Perfect Graph Theorem} states that a finite graph $G$ is perfect if and only if neither $G$ nor $\overline{G}$ contains an induced odd cycle of length at least $5$, see \cite[Theorem 14.18]{bondy}. Hence, it is easy to see that an arbitrary graph $G$ (not necessarily finite) is perfect if and only if neither $G$ nor $\overline{G}$ contains an induced odd cycle of (finite) length at least $5$. Thus, the Strong Perfect Graph Theorem is generalized to infinite graphs. Now, it is clear that a graph $G$ (not necessarily finite) is perfect if and only if $\overline{G}$ is perfect. For the finite case this is known as the \textit{Perfect Graph Theorem} verified by Lov\'{a}sz, see  \cite[Theorem 14.12]{bondy}. In the next theorem, for a zero-dimensional semi-local ring $R$, we study the perfectness of $\mathbb{CAY}(R)$. 
%In $\cite{akhtar}$, the authors studied the perfectness of $G_R$, for an Artinian ring $R$.
This theorem gives a family of infinite graphs whose every finite induced subgraph has a clique number equal to its chromatic number. 
\begin{thm}
Let $R$ be a zero-dimensional semi-local ring. Then $\mathbb{CAY}(R)$ is perfect if and only if one of the following statements holds:

\noindent {\rm{(i)}} $|{\rm {Max}}(R)|\leq 2$,\\
{\rm{(ii)}} $R$ has a residue field isomorphic to $\mathbb{Z}_2$.
\end{thm}
\begin{proof}
{We first consider the two following claims.

{\rm \bf{Claim 1.}} \textit{A graph $G$ is perfect if and only if $G/S$ is perfect}.

To see this, suppose that $G$ is perfect.  Since every induced odd cycle of length at least $5$ in $G/S$ or $\overline{G/S}$ gives an induced odd cycle of length at least $5$ in $G$ or $\overline{G}$, by the Strong Perfect Graph Theorem, we conclude that $G/S$ is perfect. Conversely, suppose that $G/S$ is perfect. By contradiction, suppose that $C$ is an induced odd cycle of length at least $5$ in $G$ or $\overline{G}$. We show that $C$ has no two distinct vertices belonging to the same equivalence class. Let $u$ and $v$ be  two distinct vertices of $C$ such that $u\thicksim v$. Since the length of $C$ is at least $5$, there exists $x\in V(C)$ adjacent to $u$ and not adjacent to $v$, a contradiction. Thus, $C$ has no two distinct vertices belonging to the same equivalence class, i.e. the existence of an induced odd cycle of length at least $5$ in $G$ or $\overline{G}$ implies an induced odd cycle of length at least $5$ in $G/S$ or $\overline{G/S}$, a contradiction. The proof of Claim $1$ is complete.

{\rm \bf{Claim 2.}} \textit{Let $n\geq 2$ be a positive integer, $G_1,\ldots,G_n$ be complete graphs (not necessarily finite) with at least two vertices and $G=G_1\times\cdots\times G_n$. If $H$ is a finite induced subgraph in $G$ or $\overline{G}$, then  there  exists positive integer $m_i$, $m_i\geq 2$, such that $H$ is an induced subgraph  either in  $K_{m_1}\times\cdots \times K_{m_n}$ or in $\overline{K_{m_1}\times\cdots \times K_{m_n}}$, respectively}. 

To see this, assume that $H$ is a finite induced subgraph of $G$. Let $M_i$ be the set of vertices of $G_i$ that appear as the $i$-th component of one vertex of $H$. Note that if $|M_j|=1$, for some $1\leq j\leq n$, then $H$ has no edge. Since $H$ is finite, every $M_i$ is a finite set of vertices of $G_i$. So, $H$ is an induced subgraph of $K_{|M_1|}\times\cdots \times K_{|M_n|}$. Thus, $m_i=|M_i|\geq 2$ are the desired integers. Now assume that $H$ is an induced subgraph of $\overline{G}$. Hence, $\overline{H}$ (the complement is respect to the complete graph on the vertices of  $G$) is a finite induced subgraph of $G$. Therefore, by the previous argument, there exists positive integer $m_i$, $m_i\geq 2$, such that $\overline{H}$ is an induced subgraph of $K_{m_1}\times\cdots \times K_{m_n}$. This implies that $H$  is an induced subgraph of  $\overline{K_{m_1}\times\cdots \times K_{m_n}}$. The proof of Claim 2 is complete. 

{\rm \bf{Claim 3.}} \textit{Let $n\geq 2$ be a positive integer and $G_1,\ldots,G_n$ be complete graphs (not necessarily finite) with at least two vertices. Then $G=G_1\times\cdots\times G_n$ is perfect if and only if either $n=2$ or $n\geq 3$ and $G_i\cong K_2$, for some $i$}.

The proof relies on \cite[Theorem A.23]{productgraphs} which characterizes perfectness of a finite direct product of finite graphs. Here, we are dealing with a finite direct product of possibly infinite graphs. To prove the claim, suppose that $G$ is perfect. If one $G_i$ has exactly two vertices, then the assertion follows. Hence, let $m_i\geq 3$ be a positive integer such that $K_{m_i}$ is a subgraph of $G_i$, for $i=1,\ldots,n$. Thus, $K_{m_1}\times\cdots\times K_{m_n}$, as a subgraph of $G$, is a finite perfect graph. So, by \cite[Theorem A.23]{productgraphs}, $n=2$. Conversely, assume that either $n=2$ or $n\geq 3$ and say $G_1\cong K_2$. Let $C$ be an induced cycle (of finite length) in $G$ or $\overline{G}$. Since $C$ is a finite graph, by Claim 2 there exists positive integer $m_i$, $m_i\geq 2$, such that $C$ is an induced cycle either in  $K_{m_1}\times\cdots \times K_{m_n}$ or in $\overline{K_{m_1}\times\cdots \times K_{m_n}}$, respectively. On the other hand $n=2$ or $n\geq 3$ and $G_1\cong K_2$. So, $n=2$ or $n\geq 3$ and $K_{m_1}\cong K_2$. Hence by applying \cite[Theorem A.23]{productgraphs} to $K_{m_1}\times\cdots \times K_{m_n}$ and noting that finite complete graphs are also complete multipartite graphs, we conclude that $K_{m_1}\times\cdots \times K_{m_n}$ is perfect. Thus, length of $C$ is less than $5$. Hence neither $G$ nor $\overline{G}$ has an induced odd cycle of length at least $5$. Hence, by Strong Perfect Graph Theorem, $G$ is perfect. The proof of Claim 3 is complete. 

Now, we prove the assertion. By Corollary \ref{modulojacobson} and Claim 1, it suffices to prove the assertion for $\mathbb{CAY}(\frac{R}{J(R)})$. By Corollary \ref{zerodivisorsofzero-dimensionals}, it is equivalent to show the assertion for $G_{\frac{R}{J(R)}}$. Let ${\rm {Max}}(R)=\{{\mathfrak{m}}_1,\ldots,{\mathfrak{m}}_n\}$, for some $n\in \mathbb{N}$. If $n=1$, then there is nothing to prove. Hence, assume that $n\geq 2$. Since by Chinese Reminder Theorem (see \cite[Proposition 1.10]{ati}), $G_{\frac{R}{J(R)}}\cong G_{\frac{R}{\mathfrak{m}_1}}\times\cdots\times G_{\frac{R}{\mathfrak{m}_n}}$, Claim 3 implies that $G_{\frac{R}{J(R)}}$ is perfect if and only if either $n=2$ or $n\geq 3$ and $R$ has a residue field isomorphic to $\mathbb{Z}_2$. The proof is complete.}
\end{proof}

By Corollary \ref{zerodivisorsofzero-dimensionals}, the following result extends Theorem 9.5 of \cite{akhtar} to  zero-dimensional semi-local rings.
\begin{cor}
Let $R$ be a zero-dimensional semi-local ring. Then $G_R$ is perfect if and only if one of the following statements holds:

\noindent {\rm{(i)}} $|{\rm {Max}}(R)|\leq 2$,\\
{\rm{(ii)}} $R$ has a residue field isomorphic to $\mathbb{Z}_2$.
\end{cor}
%%%%%%%%%%%%%%%%%%%%%%%%%%%%%%%%%%%%%%%%%%%%%%%%%%%%%%%%%%%%%%%%%%%%%%%%%%%%%%%%%%%%%%%%%%%%%%%%%%%%%%%%%%%%%%%%%%%%%%%%%%%%%%%%%%%
%%%%%%%%%%%%%%%%%%%%%%%%%%%%%%%%%%%%%%%%%%%%%%%%%%%%%%%%%%%%%%%%%%%%%%%%%%%%%%%%%%%%%%%%%%%%%%%%%%%%%%%%%%%%%%%%%%%%%%%%%%%%%%%%%%%%%%%%%%%%%%%%%%%%%%
\vspace{4mm}\noindent{\bf\large 4. The Induced Subgraph on the Regular Elements}\vspace{4mm}\\
Following \cite{totalgraph}, we are interested in studying $Reg(\mathbb{CAY}(R))$. Since the multiplication by an invertible element of $R$ is an automorphism of $Reg(\mathbb{CAY}(R))$, by Corollary \ref{zerodivisorsofzero-dimensionals}, we conclude that for every zero-dimensional ring $R$, $Reg(\mathbb{CAY}(R))$ is a vertex transitive graph. In this section, we determine the clique number and the chromatic number of $Reg(\mathbb{CAY}(R))$.

To study the coloring of $Reg(\mathbb{CAY}(R))$, we need the following theorem in which we deal with rings $R$ whose  set of zero-divisors is a union of finitely many ideals of $R$. A ring $R$ is said to \textit{have few zero-divisors} if $Z(R)$ is a union of finitely many prime ideals. By \cite{glaz}, if $R$ has
few zero-divisors, then any overring of $R$ i.e. any ring between $R$ and $T(R)$,
has few zero-divisors. In particular, by \cite[Corollary 9.36]{sharp}, any overring of a Noetherian ring has few
zero-divisors, which provides a large family of rings of this kind.
\begin{thm}\label{cliqueregular}
Let $R$ be a ring which is not an integral domain. Suppose that $|{\rm Min}(R)|<\infty$ and $Z(R)$ is a union of finitely many ideals of $R$. Then the following statements are equivalent:

\noindent {\rm{(i)}} $R$ is a finite ring,\\
{\rm{(ii)}} $\chi\big(Reg(\mathbb{CAY}(R))\big)$ is finite,\\
{\rm{(iii)}} $\omega\big(Reg(\mathbb{CAY}(R))\big)$ is finite,\\
{\rm{(iv)}} $Reg(\mathbb{CAY}(R))$ has no infinite clique.
\end{thm}
\begin{proof}
{It is clear that (i) $\Longrightarrow$ (ii), (ii) $\Longrightarrow$ (iii) and (iii) $\Longrightarrow$ (iv). We show that (iv) $\Longrightarrow$ (i). First suppose that $R$ is a non-reduced ring. Obviously, $1+Nil(R)\subseteq Reg(R)$ is a clique for $Reg(\mathbb{CAY}(R))$. Since $Reg(\mathbb{CAY}(R))$ has no infinite clique, we deduce that  $1+Nil(R)$ is finite and so $Nil(R)$ is finite. We show that $Reg(R)$ is finite as well. Let $\{r_i\}_{i=1}^{\infty}$ be an infinite subset of $Reg(R)$ and $x$ be a non-zero element of $Nil(R)$. Since $|Nil(R)|<\infty$, we conclude that there exists $A\subseteq \mathbb{N}$ such that $A$ is infinite and for every $i,j\in A$, $r_ix=r_jx$. Thus, $\{r_i\}_{i\in A}$ forms an infinite clique for $Reg(\mathbb{CAY}(R))$, a contradiction. Hence, $Reg(R)$ is finite and so by \cite[Theorem 2]{ak}, $R$ is a finite ring.
Therefore, one may assume that $R$ is a reduced ring. Let ${\rm Min}(R)=\{\mathfrak{p}_1,\ldots,\mathfrak{p}_n\}$. Thus, by Lemma \ref{minimalozerodivisors}, $Z(R)=\bigcup_{i=1}^n \mathfrak{p}_i$. We claim that for every $\varnothing \neq B\subsetneqq \{1,\ldots,n\}$,
$$\left|\bigcap_{i\in B}\mathfrak{p}_i\setminus\bigcup_{i\in B^c}\mathfrak{p}_i\right|<\infty.$$
By the Prime Avoidance Theorem (\cite[Theorem 3.61]{sharp}), $\bigcap_{i\in B^c}\mathfrak{p}_i\setminus\bigcup_{i\in B}\mathfrak{p}_i\neq \varnothing$. Let $z\in\bigcap_{i\in B^c}\mathfrak{p}_i\setminus\bigcup_{i\in B}\mathfrak{p}_i$. Now, we show that for every $x\in\bigcap_{i\in B}\mathfrak{p}_i\setminus\bigcup_{i\in B^c}\mathfrak{p}_i$, $x+z$ is a regular element of $R$. By contradiction, suppose that $x+z\in Z(R)=\bigcup_{i=1}^n \mathfrak{p}_i$. Hence, there exists $k$, $1\leq k\leq n$, such that $x+z\in \mathfrak{p}_k$. If $k\in B$, then $z\in \mathfrak{p}_k$, a contradiction. If $k\in B^c$, then $x\in \mathfrak{p}_k$, a contradiction. Thus, $\{x+z\,|\, x\in\bigcap_{i\in B}\mathfrak{p}_i\setminus\bigcup_{i\in B^c}\mathfrak{p}_i\}$ forms a clique for $Reg(\mathbb{CAY}(R))$. So, the claim is proved.
For $B$, $\varnothing \neq B\subsetneqq \{1,\ldots,n\}$, let $Z(B)=\bigcap_{i\in B}\mathfrak{p}_i\setminus\bigcup_{i\in B^c}\mathfrak{p}_i$. Since
$$Z(R)\setminus\bigcap_{i=1}^n \mathfrak{p}_i=\bigcup_{i=1}^n \mathfrak{p}_i\setminus\bigcap_{i=1}^n \mathfrak{p}_i=\bigcup_{\varnothing \neq B\subsetneqq \{1,\ldots,n\}}Z(B)$$ and $\bigcap_{i=1}^n\mathfrak{p}_i=(0)$ (see \cite[Corollary 3.54]{sharp}), we conclude that $Z(R)$ is finite. Since $R$ is not an integral domain by \cite[Theorem 2.2]{anderson-livingston}, 
%On the other hand, using that fact that $R$ is not an integral domain, we can find a non-zero element $x\in Z(R)$. Since $Z(R)$ is finite, we deduce that both $(x)$ (the ideal generated by $x$) and $Ann(x)$ are finite.  Now, by 
%$\frac{R}{Ann(x)}\cong Rx$ (as $R$-modules), 
%\cite[Theorem 2.2]{anderson-livingston}, 
we deduce that $R$ is finite and the proof is complete.}
\end{proof}

Note that by \cite[Corollary 9.36]{sharp}, Noetherian rings are among those families of rings satisfying the assumptions of the previous theorem. Also, by Lemma \ref{minimalozerodivisors} and Corollary \ref{zerodivisorsofzero-dimensionals}, reduced rings with finitely many minimal prime ideals and zero-dimensional semi-local rings are other examples of this kind of rings.
\begin{remark}
Let $R=\prod_{i\in \mathbb{N}}\mathbb{Z}_2$. It can be shown that $R$ is a zero-dimensional ring with $Reg(R)=\{1\}$ and contains infinitely many minimal prime ideals.
\end{remark}

Now, for a finite ring $R$, we would like to determine the clique number and the chromatic number of $Reg(\mathbb{CAY}(R))$. Before stating the results, we need the following notation. Let $X$ and $Y$ be two finite sets and $|X|\leq |Y|$. A \textit{Latin rectangle of size} $|X|\times |Y|$ \textit{over} $Y$, denoted by $L^{X,Y}$, is a matrix of size $|X|\times |Y|$ whose entries are in $Y$ and entries in each row and each column are distinct.
Let $(R,\mathfrak{m})$ be a finite local ring, $|\frac{R}{\mathfrak{m}}|=k$ and $\frac{R}{\mathfrak{m}}=\{f_1+\mathfrak{m},\ldots,f_k+\mathfrak{m}\}$. Let $x$ be an arbitrary element of $R$ and $f_i+\mathfrak{m}$ be the unique element of $\{f_1+\mathfrak{m},\ldots,f_k+\mathfrak{m}\}$ equal to $x+\mathfrak{m}$. We denote $f_i+\mathfrak{m}$ by $\pi(x)$ and $x-f_i$ by $\overline{x}$.
\begin{thm}\label{chromaticregulars}
Let  $R=R_1\times\cdots\times R_n$, be a finite ring, where $(R_i,\mathfrak{m}_i)$ is a local ring. If $|\frac{R_1}{\mathfrak{m_1}}|\leq\cdots\leq |\frac{R_n}{\mathfrak{m_n}}|$, then $\omega\big(Reg(\mathbb{CAY}(R))\big)=\chi\big(Reg(\mathbb{CAY}(R))\big)=|\mathfrak{m}_1|(|R_2|-|\mathfrak{m}_2|)\cdots(|R_n|-|\mathfrak{m}_n|).$
\end{thm}
\begin{proof}
{If $(R,\mathfrak{m})$ is a local ring, then by Corollary \ref{zerodivisorsofzero-dimensionals}, $Reg(\mathbb{CAY}(R)$ is a disjoint union of $|\frac{R}{\mathfrak{m}}|-1$ complete graphs $K_{|\mathfrak{m}|}$ and so $\chi\big(Reg(\mathbb{CAY}(R))\big)=\omega\big(Reg(\mathbb{CAY}(R))\big)=|\mathfrak{m}|$. So, suppose that $n\geq 2$. We make the two following claims:

\noindent {\rm\bf{Claim 1}}. $\chi\big(Reg(\mathbb{CAY}(R))\big)\leq |\mathfrak{m}_{1}|\cdots |\mathfrak{m}_{n}|\chi\big(Reg(\mathbb{CAY}(\frac{R_1}{\mathfrak{m_1}}\times \dots \times \frac{R_n}{\mathfrak{m_n}}))\big).$

Suppose that $\varphi$ is a proper vertex coloring of $Reg\big(\mathbb{CAY}(\frac{R_1}{\mathfrak{m_1}}\times \dots \times \frac{R_n}{\mathfrak{m_n}})\big)$. We define a vertex coloring $f$ of $Reg(\mathbb{CAY}(R))$ as follows:$$f\big((x_1,\ldots,x_n)\big)=\big(\overline{x_1},\ldots,\overline{x_n},\varphi(\pi_1(x_1),\ldots,\pi_n(x_n))\big).$$
Assume that $(x_1,\ldots,x_n)$ and $(y_1,\ldots,y_n)$ are two adjacent vertices with the same color in $Reg(\mathbb{CAY}(R))$. Thus, there exists $i$, $1\leq i\leq n$, such that $x_i-y_i\in \mathfrak{m}_i$. Hence, $\pi_i(x_i)=\pi_i(y_i)$. Since $\varphi\big ((\pi_1(x_1),\cdots,\pi_n(x_n))\big)=\varphi \big((\pi_1(y_1),\ldots,\pi_n(y_n))\big)$, we deduce that $\pi_j(x_j)=\pi_j(y_j)$, for every $j$, $1\leq j\leq n$. This together with $f\big((x_1,\ldots,x_n)\big)=f\big((y_1,\ldots,y_n)\big)$ implies that $(x_1,\ldots,x_n)=(y_1,\ldots,y_n)$. Thus, $f$ is a proper vertex coloring of $Reg(\mathbb{CAY}(R))$ and the claim is proved.

\noindent {\rm \bf{Claim 2}}. \textit{$\chi\big(Reg(\mathbb{CAY}(F_1\times\cdots\times F_n))\big)\leq |F_2^*|\cdots|F_n^*|$, where $n\geq 2$ and $F_i$ is a finite field with $|F_1|\leq \cdots \leq |F_n|$}.

Let $L^{F_1^*,F_i^*}$ be a Latin rectangle of size $|F_1^*|\times |F_i^*|$ over $F_i^*$, for $2\leq i\leq n$. We define a vertex coloring $g$ on $V(Reg(\mathbb{CAY}(F_1\times\cdots\times F_n)))=F_1^*\times\cdots\times F_n^*$ as follows:
$$g(x_1,\ldots,x_n)=(L^{F_1^*,F_2^*}_{x_1x_2},\ldots,L^{F_1^*,F_n^*}_{x_1x_n}),$$ where $L^{F_1^*,F_i^*}_{x_1x_i}$ denotes the $(x_1,x_i)$-entry of $L^{F_1^*,F_i^*}$. Now, suppose that $(x_1,\ldots,x_n)$ and $(y_1,\ldots,y_n)$ are two distinct adjacent vertices with the same color in $Reg(\mathbb{CAY}(F_1\times\cdots\times F_n))$. Hence, there exists $t$, $1\leq t\leq n$, such that $x_t=y_t$. Since $L^{F_1^*,F_i^*}_{x_1x_i}=L^{F_1^*,F_i^*}_{y_1y_i}$,
for every $i$, $2\leq i\leq n$, we deduce that $x_1=y_1$. Therefore, $L^{F_1^*,F_i^*}_{x_1x_i}=L^{F_1^*,F_i^*}_{y_1y_i}$
for every $i$, $2\leq i\leq n$. This implies that $(x_1,\ldots,x_n)=(y_1,\ldots,y_n)$, a contradiction. So, $g$ is a proper vertex coloring of $Reg(\mathbb{CAY}(F_1\times\cdots\times F_n))$ and the claim is proved.

Since $|\frac{R_1}{\mathfrak{m_1}}|\leq\cdots\leq |\frac{R_n}{\mathfrak{m_n}}|$, it follows by these two claims that $\chi\big(Reg(\mathbb{CAY}(R))\big)\leq|\mathfrak{m}_1|(|R_2|-|\mathfrak{m}_2|)\cdots(|R_n|-|\mathfrak{m}_n|)$. Since $(1+\mathfrak{m}_1)\times(R_2\setminus \mathfrak{m}_2)\times\cdots\times (R_n\setminus \mathfrak{m}_n)$ forms a clique for $Reg(\mathbb{CAY}(R))$, we conclude that $\omega\big(Reg(\mathbb{CAY}(R))\big)\geq|\mathfrak{m}_1|(|R_2|-|\mathfrak{m}_2|)\cdots(|R_n|-|\mathfrak{m}_n|)$
and so the assertion is proved.}
\end{proof}
%%%%%%%%%%%%%%%%%%%%%%%%%%%%%%%%%%%%%%%%%%%%%%%%%%%%%%%%%%%%%%%%%%%%%%%%%%%%%%%%%%%%%%%%%%%%%%%%%%%%%%%%%%%%%%%%%%%%
%%%%%%%%%%%%%%%%%%%%%%%%%%%%%%%%%%%%%%%%%%%%%%%%%%%%%%%%%%%%%%%%%%%%%%%%%%%%%%%%%%%%%%%%%%%%%%%%%%%%%%%%%%%%%%%%%%%%%%%%%%%%%%%%%%%
%%%%%%%%%%%%%%%%%%%%%%%%%%%%%%%%%%%%%%%%%%%%%%%%%%%%%%%%%%%%%%%%%%%%%%%%%%%%%%%%%%%%%%%%%%%%%%%%%%%%%%%%%%%%%%%%%%%%%%%%%%%%%%%%%%%%%%%%%%%%%%%%%%%%%%
\noindent{\bf Acknowledgements.} The authors are indebted to the School of Mathematics, Institute for
Research in Fundamental Sciences, (IPM), for support. The research
of the second author was in part supported by a grant from IPM (No.
90050212).
%%%%%%%%%%%%%%%%%%%%%%%%%%%%%%%%%%%%%%%%%%%%%%%%%%%%%%%%%%%%%%%%%%%%%%%%%%%%%%%%%%%%%%%%%%%%%%%%%%%%%%%%
%%%%%%%%%%%%%%%%%%%%%%%%%%%%%%%%%%%%%%%%%%%%%%%%%%%%%%%%%%%%%%%%%%%%%%%%%%%%%%%%%%%%%%%%%%%%%%%%%%%%%%%%

{}


\begin{thebibliography}{}{\small

\bibitem{Aali} G. Aalipour, S. Akbari, Some properties of a Cayley graph of a commutative ring, Comm. Algebra, to appear.

\bibitem{ak} S. Akbari, D. Kiani, F. Mohammadi, S. Moradi, The total graph and regular graph of a commutative ring, J. Pure Appl. Algebra 213 (12) (2009) 2224--2228.

%\bibitem{akb} S. Akbari, A. Mohammadian,  On zero-divisor graphs of finite
%rings, J. Algebra 314 (2007) 168--184.

\bibitem{akhtar} R. Akhtar, M. Boggess, T. Jackson-Henderson, I. Jim$\text{\`{e}}$nez, R. Karpman, A. Kinzel, D. Pritikin, On the unitary Cayley graph of a finite ring, Electron. J. Combin. 16 (2009), no. 1, \#R117.

\bibitem{totalgraph} D.F. Anderson, A. Badawi, The total graph of a commutative ring, J. Algebra 320 (7) (2008) 2706--2719.

%\bibitem{anderson-levy} D.F. Anderson, R. Levya, J. Shapiro, Zero-divisor graphs, von Neumann regular rings, and Boolean algebras, J. Pure Appl. Algebra 180 (3) (2003) 221--241.

\bibitem{anderson-livingston} D.F. Anderson, P.S. Livingston, The zero-divisor graph of a commutative ring, J. Algebra 217 (2) (1999) 434--447.

\bibitem{ati} M.F. Atiyah, I.G. Macdonald, Introduction to
Commutative Algebra, Addison-Wesley Publishing Company, 1969.

\bibitem{japanese} K. Baba, J. Sato, The chromatic number of the simple graph associated with a commutative ring, Sci. Math. Jpn. 71 (2) (2010) 187--194.

\bibitem{Basic-Illic-clique} M. Ba$\text{\v{s}}$i$\text{\'{c}}$, A. Ili$\text{\'{c}}$, On the clique number of integral circulant graphs, Appl. Math. Lett. 22 (2009) 1406--1411.

%\bibitem{belshoff} R. Belshoff, J. Chapman, Planar zero-divisor graphs, J. Algebra 316 (2007) 471--480.

%\bibitem{berrizbeitia-giudici-old} P. Berrizbeitia, R.E. Giudici, Counting pure $k$ cycles in sequences of Cayley graphs,
%Discrete Math. 149 (1996) 11--18.

%\bibitem{berrizbeitia-giudici} P. Berrizbeitia, R.E. Giudici, On cycles in the sequence of unitary Cayley graphs,
%Discrete Math. 282 (2004) 1--3.

\bibitem{bondy} J. A. Bondy, U. S. R. Murty, Graph Theory, Graduate Texts in Mathematics, 244 Springer, New York, 2008.

%\bibitem{dejter-giudici} I. Dejter and R. E. Giudici, On unitary Cayley graphs, J. Combin. Math. Combin.
%Comput. 18 (1995) 121--124.

\bibitem{fuchs} E.D. Fuchs,  Longest induced cycles in circulant graphs, Electron. J. Combin. 12 (2005), \#R52.

\bibitem{glaz} S. Glaz, Controlling the zero divisors of a commutative ring. (English summary) Commutative ring theory and applications, 191--212,
Lecture Notes in Pure and Appl. Math., 231, Dekker, New York, 2003.

\bibitem{ilic} A. Ili$\text{\'{c}}$, The energy of unitary Cayley graphs, Linear Algebra Appl. 431 (2009) 1881--1889.

\bibitem{Basic-Illic-chromatic} A. Ili$\text{\'{c}}$, M. Ba$\text{\v{s}}$i$\text{\'{c}}$, On the chromatic number of integral circulant graphs, Computers and Mathematics with Applications 60 (2010) 144--150.

\bibitem{productgraphs} W. Imrich, S. Klav$\check{\text{z}}$ar, Product Graphs, Structure and Recognition, Wiley-Interscience, New York, 2000.


\bibitem{kaplansky} I. Kaplansky, Commutative Rings, rev. ed., University of Chicago Press, Chicago, 1974.

\bibitem{kiani} D. Kiani, M.M.H. Aghaei, Y. Meemark, B. Suntornpoch, Energy of unitary Cayley graphs and gcd-graphs,
Linear Algebra Appl. 435 (2011) 1336--1343.

\bibitem{klotz-sander} W. Klotz, T. Sander, Some properties of unitary Cayley graphs, Electron. J. Combin. 14 (2007), no. 1, \#R45.




\bibitem{liu-zhou} X. Liu, S. Zhou, Spectral properties of unitary Cayley graphs of finite commutative rings, Electron. J. Combin. 19 (4) (2012), \#P13.

\bibitem{marusic} D. Maru$\check{\text{s}}$i$\check{\text{c}}$, Hamiltonian circuits in Cayley graphs, Discrete Math. 46 (1) (1983) 49--54.


\bibitem{sharp} R.Y. Sharp, Steps in Commutative Algebra, Second edition, Cambridge University Press, 2000.

\bibitem{shekarriz} M.H. Shekarriz, M.H. Shirdareh Haghighi, H. Sharif, On the total graph of a finite commutative ring, Comm. Algebra 40 (8) (2012), 2798--2807. 

\bibitem{so} W. So, Integral circulant graphs, Discrete Math. 306 (2006) 153--158.

\bibitem{spacapan} S. $\check{\text{S}}$pacapan, Connectivity of Cartesian products of graphs, Applied Mathematics Letters 21 (2008) 682--685.

\bibitem{watkins} M.E. Watkins. Connectivity of transitive graphs. J. Combin. Theory 8 (1970) 23--29.

\bibitem{west} D.B. West, Introduction to Graph Theory, 2nd ed., Prentice Hall, 2001.



}
\end{thebibliography}
\end{document}